\documentclass{amsart}
\usepackage[utf8]{inputenc}
\usepackage{amssymb}
\usepackage{amsthm}
\usepackage{amsmath}
\usepackage{graphicx,wrapfig,animate}
\usepackage[unicode, pdftex]{hyperref}
\hypersetup{
    colorlinks=true,
    linkcolor=blue,
    filecolor=magenta,      
    urlcolor=blue,
    pdfpagemode=FullScreen,
}

\newtheorem{theorem}{Theorem}[section]

\newtheorem{corollary}[theorem]{Corollary}
\newtheorem{conjecture}[theorem]{Conjecture}
\newtheorem{observation}[theorem]{Observation}

\begin{document}
\title[Majority Dynamics and Internal Partitions]{Majority Dynamics and Internal Partitions of Random Regular Graphs: Experimental Results}
\author[P. Arkhipov]{Pavel Arkhipov}
\date{}

\begin{abstract}
    This paper focuses on Majority Dynamics in sparse graphs, in particular, as a tool to study internal cuts. It is known that, in Majority Dynamics on a finite graph, each vertex eventually either comes to a fixed state, or oscillates with period two. The empirical evidence acquired by simulations suggests that for random odd-regular graphs, approximately half of the vertices end up oscillating with high probability. We notice a local symmetry between oscillating and non-oscillating vertices, that potentially can explain why the fraction of the oscillating vertices is concentrated around $\frac{1}{2}$. In our simulations, we observe that the parts of random odd-regular graph under Majority Dynamics with high probability do not contain $\lceil \frac{d}{2} \rceil$-cores at any timestep, and thus, one cannot use Majority Dynamics to prove that internal cuts exist in odd-regular graphs almost surely. However, we suggest a modification of Majority Dynamics, that yields parts with desired cores with high probability.
\end{abstract}

\maketitle

\bigskip

\section{Introduction}

Let $G = (V, E)$ be an $n$-vertex graph. A partition of $G$ into sets $S$ and $V \setminus S = \overline{S}$ is an \textit{internal cut} if for every vertex the number of its neighbors in its own part is not less than the number of its neighbors in the other part. 

For a given partition, let us say that a vertex is \textit{unsatisfied} if the number of its neighbors in the opposite part is greater than the number of its neighbors in its own part and \textit{satisfied} otherwise. We will also call a set $S \subset V$ satisfied if all the vertices in $S$ are satisfied.

It is known that for $d = 3$, the only two graphs with no internal partition are $K_4$ and $K_{3, 3}$ \cite{Shafique_Dutton}. Also, for $d = 6$ any graphs with at least 12 vertices has an internal partition \cite{Ban_Linial}. It is conjectured that for all $d$ there is only a finite number of graphs with no internal cuts, however, the question remains open for $d \notin \{2, 3, 4, 6\}$.

If $d$ is even, then a random $d$-regular graph admits an internal partition with high probability \cite{even_d_whp}.

In \cite{partitioning_via_random_processes} it is shown that, for sufficiently large $d$-regular graphs, there are partitions such that almost every vertex is satisfied.

One of the tools that can be used to study internal partitions is Majority Dynamics. Encode a partition with a vector $x \in \{-1, 1\}^n$. We will call the corresponding parts of this cut \textit{positive} and \textit{negative}. Consider the following update rule:
\begin{equation}
    x_i^{new} = 
    \begin{cases}
        - 1, & \sum\limits_{j \sim i} x_j < 0 \\
        x_i, & \sum\limits_{j \sim i} x_j = 0 \\
        1, & \sum\limits_{j \sim i} x_j > 0 \\
    \end{cases}.
\end{equation}
In words, the $i$-th vertex switches parts if the strict majority of its neighbors are in another part of the cut. Starting from some initial configuration $x^0$, this update rule produces a sequence of states $(x^0, x^1, \ldots, x^t, \ldots)$.

It is known that such dynamics eventually converges to a cycle of length at most two \cite{period_2}, so there exists time $T$ such that $x^t = x^{t+2}$ for $t > T$. One may notice that if Majority Dynamics converges to a fixed point with non-trivial positive and negative parts, then this fixed point corresponds to an internal cut. 

It has been proved in \cite{md_paper} that if $G$ is a random $d$-regular graph or an Erdos-Renyi random graph with average degree $d$, then for every $\varepsilon > 0$ there exists a time $t$ such that, with high probability, $x^{t+2}_i = x^t_i$ for all $i$ except a set of size $\varepsilon n$, and the fraction of nodes for which $x^t_i = 1$ is, with high probability, in $[1/2 - \varepsilon, 1/2 + \varepsilon]$, if the initial cut was chosen uniformly at random.

\section{Oscillating vertices in Majority Dynamics}

Consider a final at-most-2-cycle of Majority Dynamics: $x^t = x^{t+2}$. In this limit cycle, a vertex $i$ either stays in a positive or negative part forever (so $x^t_i = x^{t+1}_i$), or oscillates between the parts ($x^t_i \neq x^{t+1}_i$). A gallery of random examples of final states for small regular graphs is given in Figure \ref{fig_gallery}. Oscillating vertices are highlighted with thick boundaries. As one can see, in some cases, the unanimity is reached, so the final state either every vertex is positive or every vertex is negative. Sometimes, every vertex in the graph oscillates. Sometimes, only a part of the graph oscillates.

\begin{figure}[!h]
    \centering
    \includegraphics[width=1.\textwidth]{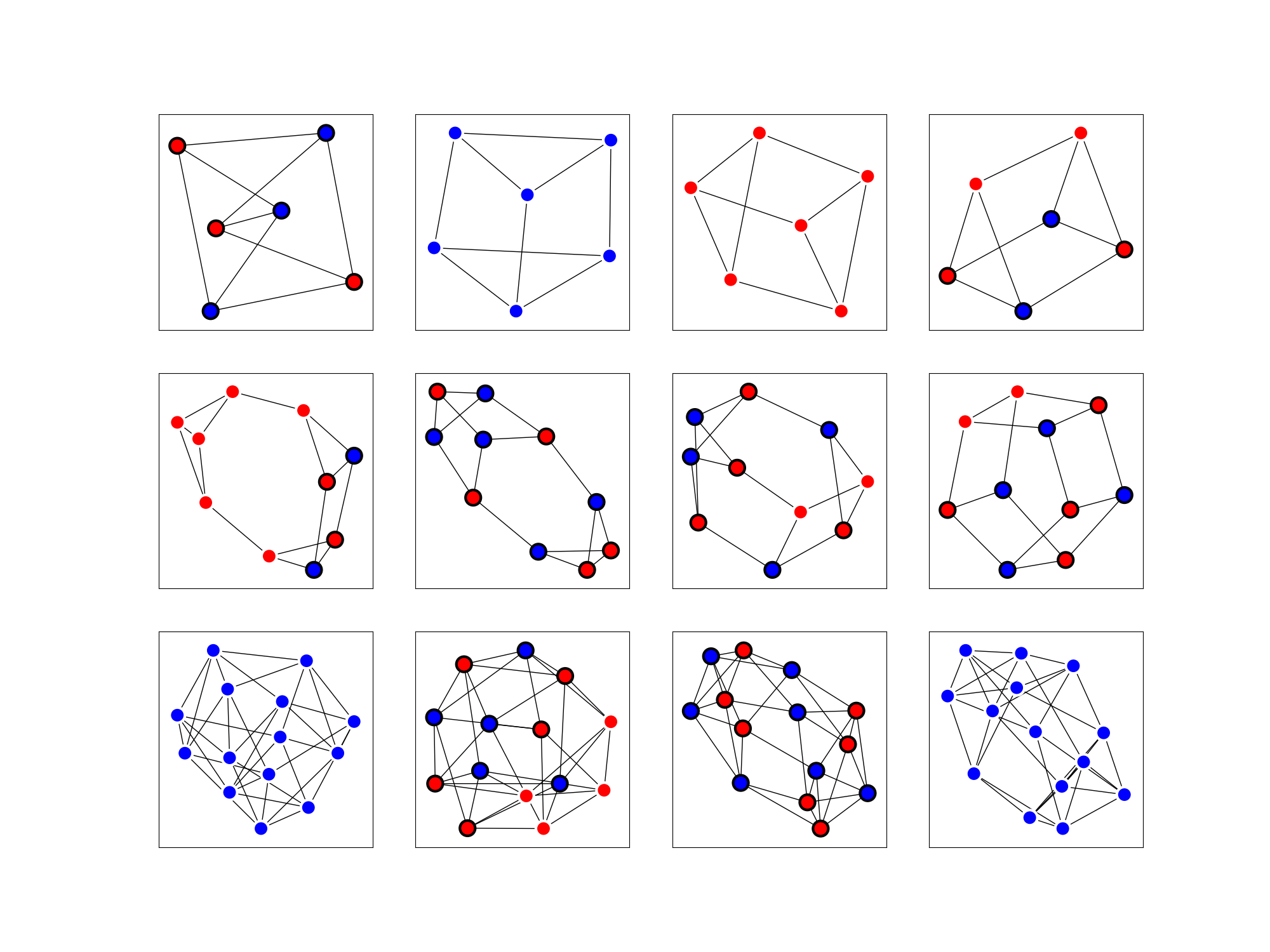}
    \caption{Examples of final configurations for some 3- and 5-regular graphs. Red vertices are in the positive part, blue vertices are in the negative part. Vertices with thick black boundary oscillate.}
    \label{fig_gallery}
\end{figure}

Let us consider a fraction of vertices that oscillate in the final configuration for $d$-regular graphs as $n \to \infty$. The histograms for the fraction of oscillating vertices are presented in Figure \ref{fig_histograms_regular}. As we can see, for odd $d$, the oscillating fraction gets concentrated around $\frac{1}{2}$, and for even $d$, it gets concentrated around zero.

The code used to generate all the data is available at GitHub, \\ \href{https://github.com/Paul566/internal-cuts}{https://github.com/Paul566/internal-cuts}.

\begin{figure}[!h]
    \centering
    \includegraphics[width=0.85\textwidth]{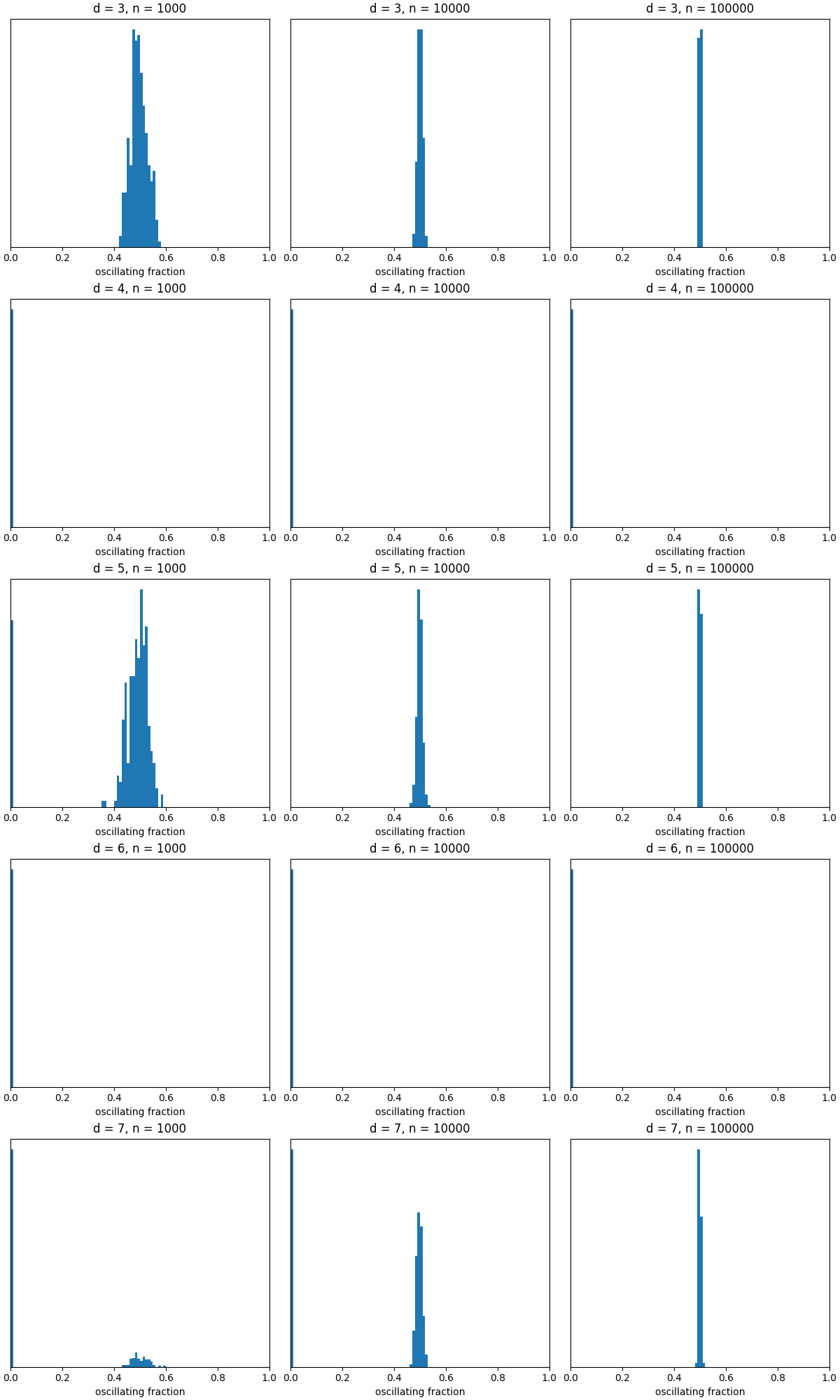}
    \caption{Histograms for the fraction of oscillating vertices for $d$-regular graphs. Cases for $d \in \{3, 4, 5, 6, 7\}$, $n \in \{10^3, 10^4, 10^5\}$.}
    \label{fig_histograms_regular}
\end{figure}

The numerical experiments suggest the following conjecture:
\begin{conjecture}
    In the limiting cycle of the Majority Dynamics for a random $d$-regular graph with a random initial configuration, the number of oscillating vertices is, with high probability, in $[(1/2 - \varepsilon) n, (1/2 + \varepsilon) n]$ if $d$ is odd, and in $[0, \varepsilon n]$ if $d$ is even.
    \label{conj_half_oscillates}
\end{conjecture}

We will now describe a symmetry in dynamics that makes the concentration of the oscillating fraction around $\frac{1}{2}$ believable.

For a given partition $x$, color an edge $(i, j)$ red if $x_i \neq x_j$, and blue if $x_i = x_j$. Let's describe the dynamics in terms of the red and blue subgraphs instead of the positive and negative parts. We will denote the degree of a vertex $i$ in the red and blue subgraphs as $\deg_R(i)$ and $\deg_B(i)$ respectively. Such a coloring of the edges of $G$ encodes the current partition up to a symmetry of switching signs of $x_i$ in some connected components of $G$. A given vertex $i$ is unsatisfied if and only if the strict majority of the edges adjacent to it are red. During a step of Majority Dynamics, an edge $(i, j)$ changes its color if and only if exactly one of the vertices $i, j$ is unsatisfied. So, the update rule for the red and blue subgraphs are the following:

\begin{itemize}
    \item Edge (i, j) changes its color if either ($\deg_R(i) > \deg(i) / 2$ and $\deg_R(j) \leq \deg(j) / 2$) or ($\deg_R(i) \leq \deg(i) / 2$ and $\deg_R(j) > \deg(j) / 2$),

    \item Edge (i, j) stays the same color if either ($\deg_R(i) > \deg(i) / 2$ and $\deg_R(j) > \deg(j) / 2$) or ($\deg_R(i) \leq \deg(i) / 2$ and $\deg_R(j) \leq \deg(j) / 2$).
\end{itemize}

\begin{observation}
    If $\deg(i)$ and $\deg(j)$ are odd, then the update rules are symmetrical with respect to switching red and blue colors for all edges. 
\end{observation}
\begin{proof}
    Each edge is either red or blue, so $\deg_R(i) + \deg_B(j) = \deg(j)$. Then, the condition of exactly one of $i, j$ having the majority of edges red is equivalent to the condition of exactly one of $i, j$ having the majority of edges blue.
\end{proof}

There is the following intuition behind the odd degree case in Conjecture \ref{conj_half_oscillates}.

In general, globally, the red and blue subgraphs are different, because the red subgraph has to be bipartite, and the blue subgraph has to be disconnected unless the red subgraph is empty. However, locally a random $d$-regular graph looks like a tree with high probability. Any subgraph of a tree is bipartite, and its complement is disconnected unless it is empty. Therefore, with high probability, one cannot distinguish the red subgraph and the blue subgraph by looking at them in a neighborhood of some vertex. The oscillating vertices are the vertices with the majority of the neighboring edges red, and the non-oscillating vertices are the vertices with the majority of the neighboring edges blue. Since the degrees are odd, there cannot be any ties. Since the red and the blue subgraphs behave symmetrically during the Majority Dynamics, then around half of the vertices should be unsatisfied at any time step.

What happens if some of the vertices have even degrees?

\begin{observation}
    In the limit cycle, each leaf of the subgraph generated by vertices that oscillate has odd degree.
\end{observation}
\begin{proof}
    Assume the contrary. Suppose, a vertex $i$ of even degree $d$ is oscillating, it has 1 oscillating neighbor $j$ and $d - 1$ non-oscillating neighbors. Since $i$ oscillates, at each sufficiently large time $t$ it has at least $\frac{d}{2} + 1$ red adjacent edges. There are two cases.
    \begin{enumerate}
        \item The edge $(i, j)$ is red. Then there are at least $\frac{d}{2}$ edges (out of $d - 1$) that lead from $i$ to non-oscillating vertices. At the next step, the colors of those edges will switch, so $i$ will have at most $d - 1 - \frac{d}{2} + 1 = \frac{d}{2}$ red edges adjacent in total, and $i$ will stop oscillating, which is a contradiction.

        \item The edge $(i, j)$ is blue. Then there are at least $\frac{d}{2} + 1$ edges that lead from $i$ to non-oscillating vertices. At the next step, $i$ will have at most $d - 1 - \frac{d}{2} = \frac{d}{2} - 1$ red edges adjacent in total, and $i$ will cease to oscillate, which is a contradiction.
    \end{enumerate}
\end{proof}

\begin{corollary}
    If $d$ is even, then the subgraph of any $d$-regular graph generated by oscillating vertices does not have leaves.
\end{corollary}

Then, a component of the oscillating subgraph in an even-regular graph has to contain cycles. Since a random regular graphs is locally a tree with high probability, a given vertex is almost surely not inside a component of the oscillating subgraph of bounded diameter.

In addition, let us consider the dynamics on the Erdos-Renyi graphs $G(n, d / n)$. The histograms for the fraction of oscillating vertices is in Figure \ref{fig_histograms_er}. It looks like for relatively small $n$, there are common cases of nearly empty oscillating part, but for large $n$, those cases are unlikely.


\begin{figure}[h]
    \centering
    \includegraphics[width=1.\textwidth]{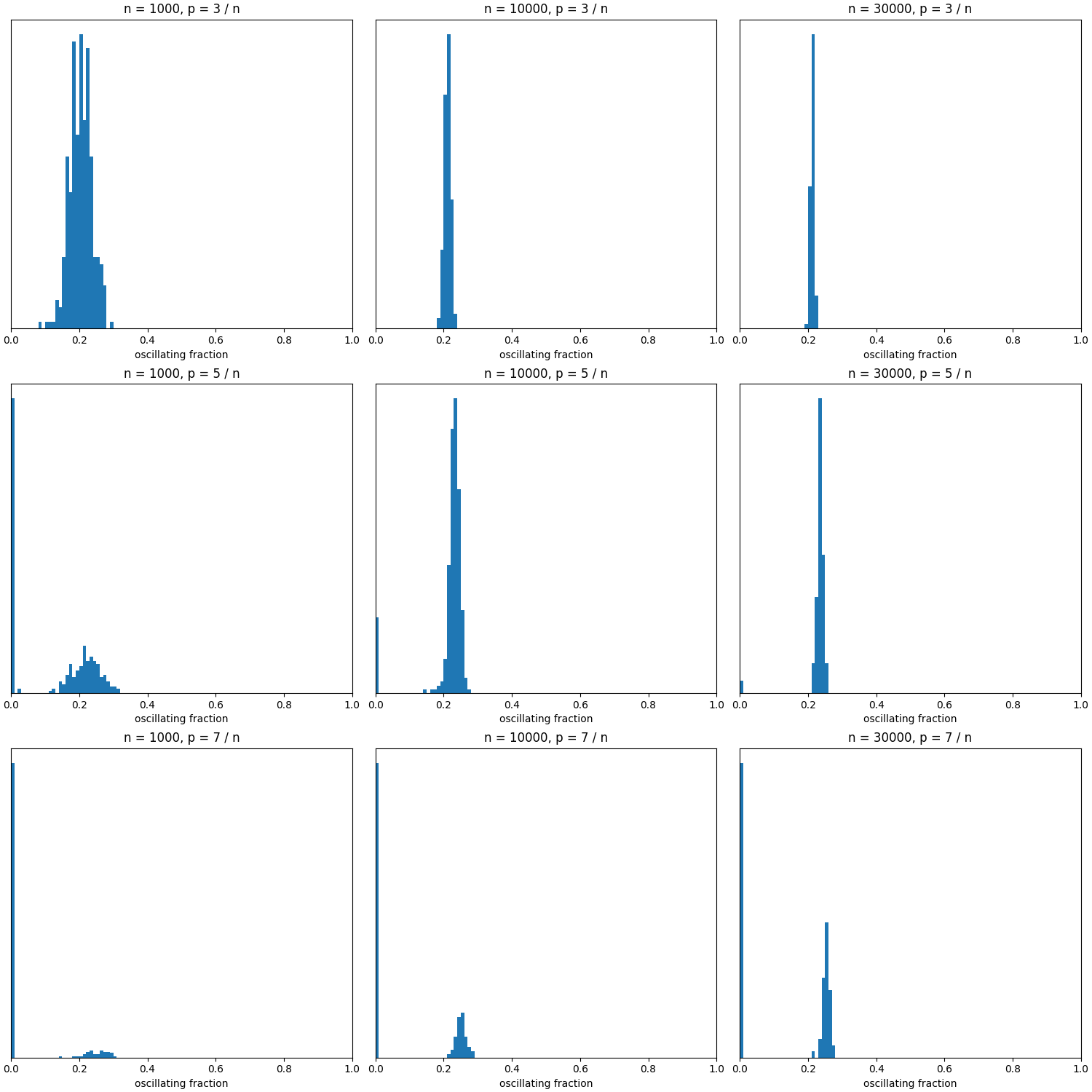}
    \caption{Histograms for the fraction of oscillating vertices for $G(n, d / n)$. Cases for $d \in \{3, 5, 7\}$, $n \in \{10^3, 10^4, 3 \cdot 10^4\}$.}
    \label{fig_histograms_er}
\end{figure}

\section{Majority Dynamics and Majority Dynamics with zeros for studying internal partitions}


We now consider the applications of Majority Dynamics in studying internal cuts of regular graphs. Let us say that a subgraph of a $d$-regular graph is a core if all vertices in this subgraph have degrees at least $\frac{d}{2}$. Clearly, if a graph has two disjoint cores, then there is an internal cut in such a graph. 

Since the Majority Dynamics is designed in a way that unsatisfied vertices go to the side where they (probably) be satisfied, there is a natural hope that at some point of Majority Dynamics cores will appear in positive and negative parts. The probability that a core emerges in the positive part at any point is illustrated in the Figure \ref{fig_positive_core}. In the simulations, we generate a random $d$-regular graph, then we generate the initial configuration uniformly at random and run Majority Dynamics until it converges to a cycle of length at most 2. Then we keep looking for unsatisfied vertices in the positive part of the cut and include them in the negative part one by one. The positive part is non-empty when this process stops if and only if there is a core in the positive part of the cut.
\begin{figure}[h]
    \centering
    \includegraphics[width=0.9\textwidth]{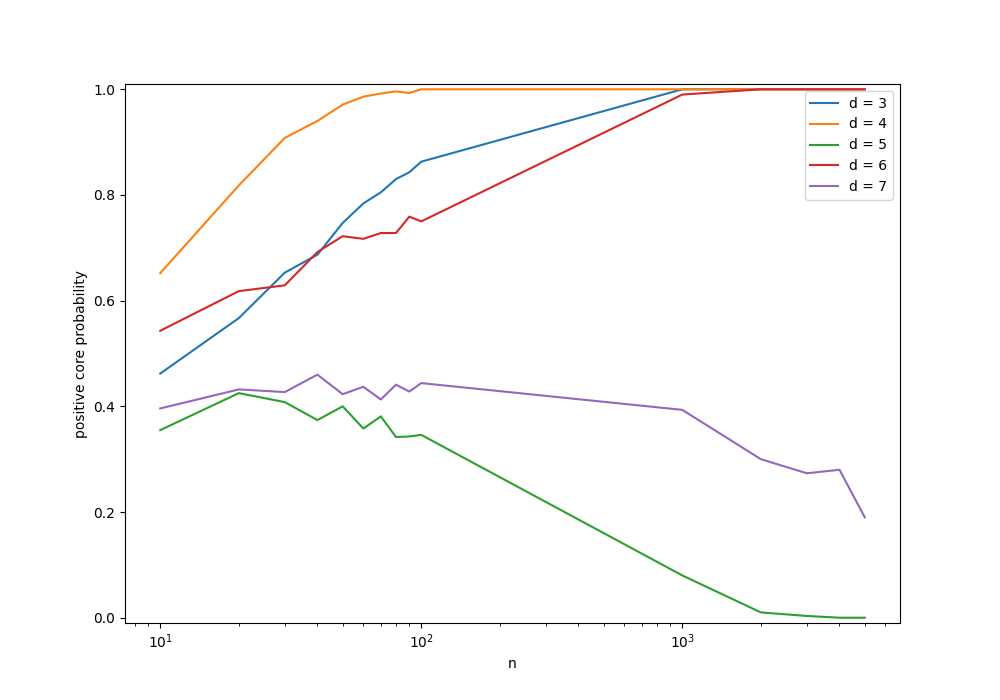}
    \caption{Probability of positive core after reaching the limit cycle of Majority Dynamics.}
    \label{fig_positive_core}
\end{figure}

Notice that it is enough to check if such a core exists when Majority Dynamics converged to a cycle of length at most two, since a core cannot disappear once it emerged. If $d$ is even or $d = 3$, the probability tends to 1 as $n$ goes to infinity, and for odd $d$ greater than 3, the probability tends to 0. The case $d = 3$ is unique, because then any cycle is a core, and a positive core exists with high probability right after the initialization. 

The data suggests that trying to find cores in positive and negative parts during Majority Dynamics should not yield new results of the form ``For this specific $d$, random $d$-regular graphs have an internal cut with high probability'', as such facts are already established for even $d$ and $d = 3$.


If Majority Dynamics cannot create an internal cut with high probability for odd $d > 3$, then what kind of dynamics can? We consider two possible answers.

First, it makes sense to look at the swap process, which picks a random unsatisfied vertex and makes it switch parts. It is different from Majority Dynamics, which made all unsatisfied vertices switch parts simultaneously at each step. Unlike Majority Dynamics, the swap process converges to a fixed point, because the size of the cut strictly decreases at each step. Therefore, if the swap process terminates at a non-trivial partition, such partition has to be internal. The plot of the probability that the swap process converges to an internal partition starting from a random partition for a random $d$-regular graph is in Figure \ref{fig_internal_cut_1_by_1}.

\begin{figure}[h]
    \centering
    \includegraphics[width=0.9\textwidth]{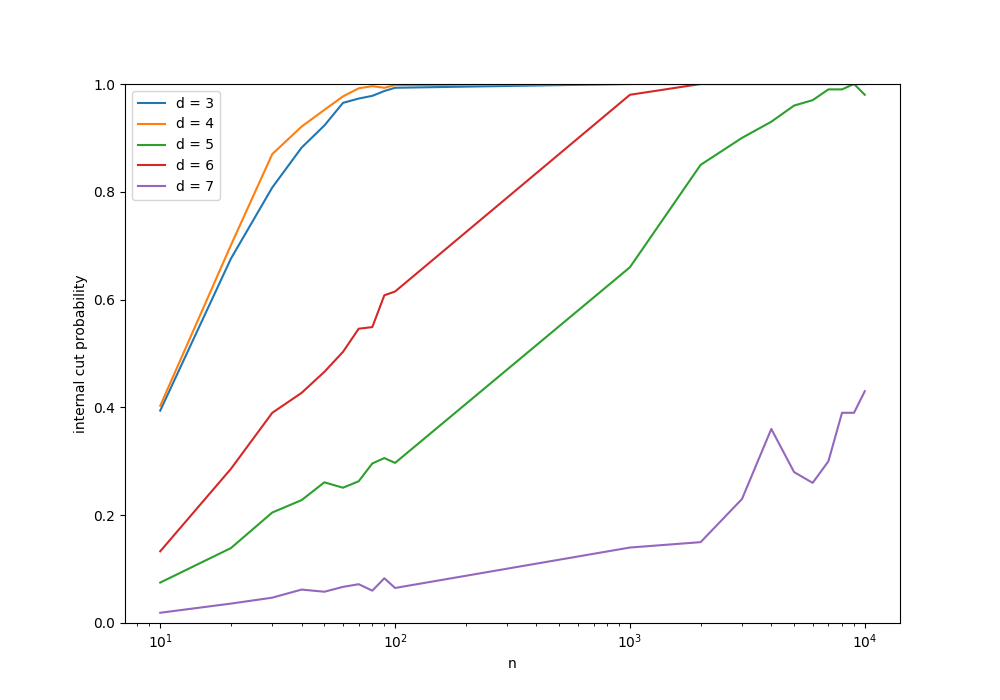}
    \caption{Probability to converge to an internal cut during the swap process.}
    \label{fig_internal_cut_1_by_1}
\end{figure}

This plot suggests that for regular graphs, internal cuts do exist with high probability and can be obtained with one run of the swap process from a random initial configuration. In the specific case of $d = 5$, $n = 10^4$, the thresholding amount of steps after which a core emerges seems to be around $0.3 \cdot n$.

\begin{figure}[h]
    \centering
    \includegraphics[width=0.9\textwidth]{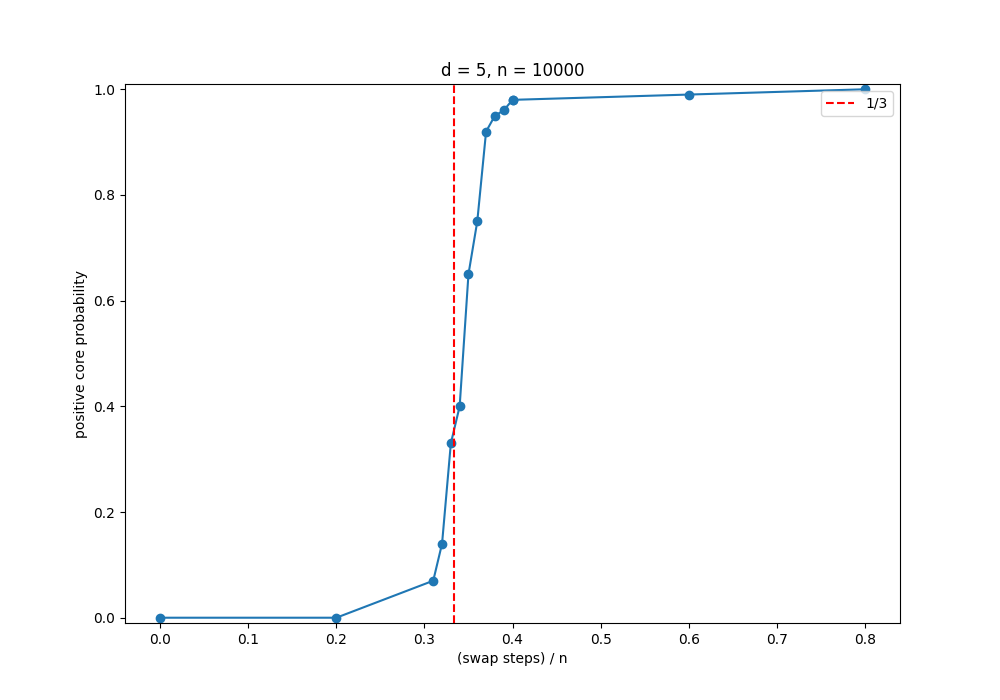}
    \caption{Probability to have a positive core, depending on the number of the swap process steps.}
    \label{fig_proba_red_core_swap}
\end{figure}

However, the one-by-one processes are harder to analyze theoretically than Majority Dynamics-like processes.

Let us introduce Majority Dynamics With Zeros. The graph will now have three parts: positive, negative and neutral. The corresponding state is $x \in \{-1, 0, 1\}^n$. The update rule is the following:
\begin{equation}
    x_i^{new} = clip \left( x_i + \text{sgn} \left( \sum\limits_{j \sim i} x_j \right) \right),
\end{equation}
where 
\begin{equation}
    clip(t) = 
    \begin{cases}
        -1, & t \leq -1 \\
        1, & t \geq 1 \\
        t, & else
    \end{cases}.
\end{equation}

In words, if $i$ is negative and has more positive than negative neighbors, then it becomes neutral. The symmetrical rule holds for positive vertices. If a neutral vertex has more neighbors in the positive (negative) part than in the negative (positive) part, it becomes positive (negative). From now on, we will call Majority Dynamics With Zeros MD0.

One may notice that zero is a fixed point for MD0. However, this fixed point is unstable in the sense that if all $x_i$ except $x_1$ are equal to zero, then eventually all the vertices in the connected component of vertex $1$ will move to the initial part of vertex $1$.

We would like a core to emerge in the (without loss of generality) positive part during MD0. To test this, we generate a random $d$-regular graph and a random initial configuration with empty neutral part, so $x^0 \in \{-1, 1\}^n$. We run MD0 for $k$ steps. Then, we set every neutral vertex to be negative and search for a positive core as before, greedily switching unsatisfied positive vertices. The results for $d = 5$ and $d = 7$ are presented in Figures \ref{fig_core_MD0_5} and \ref{fig_core_MD0_7}. 

\begin{figure}[!h]
    \centering
    \includegraphics[width=0.9\textwidth]{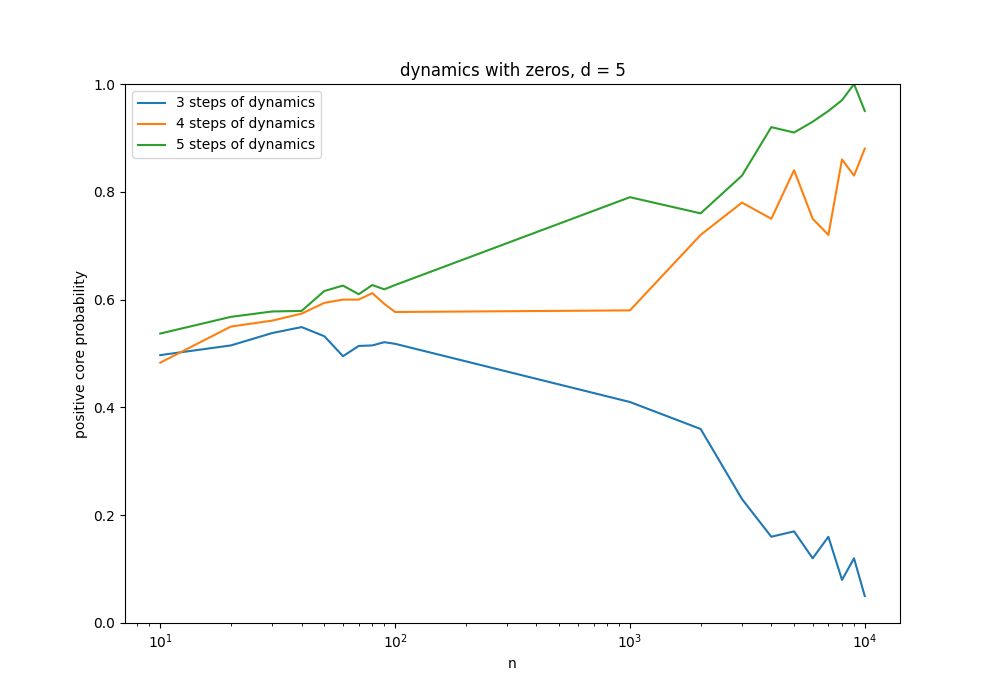}
    \caption{Probability to have a positive core after $k$ steps of MD0 for random $5$-regular graphs.}
    \label{fig_core_MD0_5}
\end{figure}
\begin{figure}[!h]
    \centering
    \includegraphics[width=0.9\textwidth]{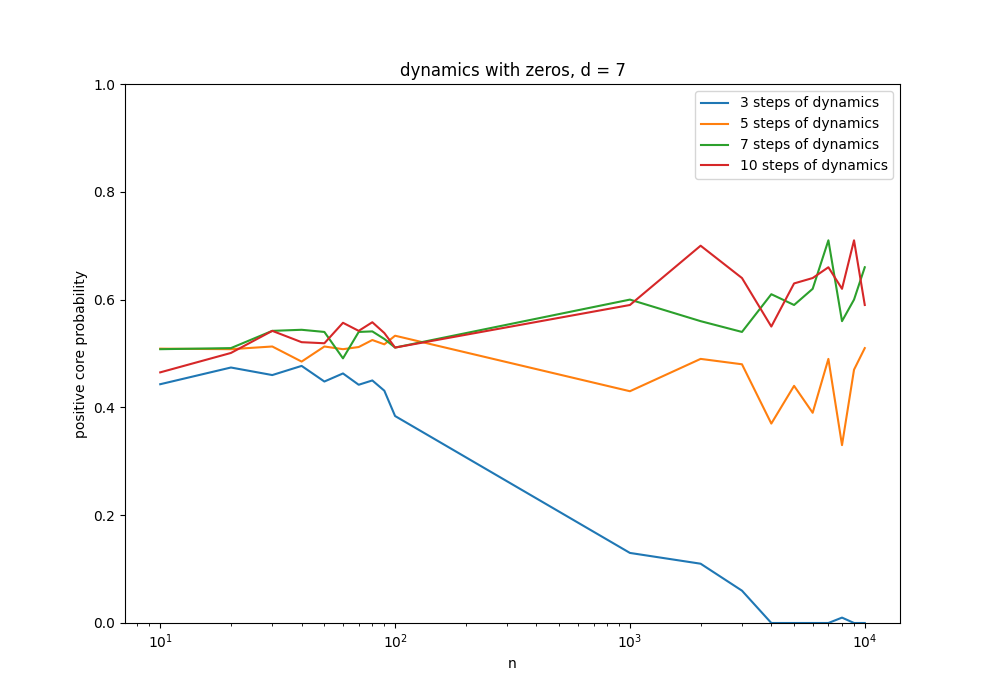}
    \caption{Probability to have a positive core after $k$ steps of MD0 for random $7$-regular graphs.}
    \label{fig_core_MD0_7}
\end{figure}

The plots suggest that in case $d = 5$, a core emerges inside the positive part with high probability after 4 steps of MD0. We believe that this can be proved with Warning Propagation technique \cite{warning_propagation}. We see the following approach: split the vertices into $3^5 = 243$ types depending on their values $(x_i^0, x_i^1, x_i^2, x_i^3, x_i^4)$. Then, search for a nontrivial core inside $\{i \rvert x_i^4 = 1\}$ using these types. The existence of a non-empty core in the positive part with high probability would imply the existence of a non-empty core in the negative part with high probability by symmetry. Then, by the union bound, there are cores in both positive and negative parts, and thus, it would imply that the internal cut exists in random 5-regular graphs with high probability.

\section{Acknowledgements}
I am grateful to Matthew Kwan for setting the problem, providing useful literature, fruitful discussions, text review, mentorship, general encouragement and support.

\end{document}